\documentclass[12pt]{elsarticle}
\usepackage{amsmath}
\usepackage{amssymb}
\usepackage{amsfonts}
\usepackage{amsthm}
\usepackage{setspace}
\usepackage[margin=3cm]{geometry}
\usepackage{lscape}
\usepackage{slashbox}
\usepackage{subfig}
\usepackage{url}

\newtheorem{theorem}{Theorem}
\newtheorem{lemma}[theorem]{Lemma}

\newcommand{\norm}[1]{\left \lVert #1 \right \rVert}
\newcommand{\abs}[1]{\left \lvert {#1} \right \rvert}
\newcommand{\ip}[2]{\left \langle {#1},{#2} \right\rangle}

\usepackage{color}
\usepackage{xparse}

\begin{document}

\begin{frontmatter}

\title{Sampling and Approximation of Bandlimited Volumetric Data}

\author[address]{Rami Katz}
\ead{rami@benis.co.il}

\author[address]{Yoel Shkolnisky\corref{correspondingauthor}}
\cortext[correspondingauthor]{Corresponding author}
\ead{yoelsh@post.tau.ac.il}

\address[address]{Department of Applied Mathematics, School of Mathematical Sciences, Tel-Aviv University, Israel}

\begin{abstract}

We present an approximation scheme for functions in three dimensions, that requires only their samples on the Cartesian grid, under the assumption that the functions are sufficiently concentrated in both space and frequency. The scheme is based on expanding the given function in the basis of generalized prolate spheroidal wavefunctions, with the expansion coefficients given by weighted dot products between the samples of the function and the samples of the basis functions. As numerical implementations require all expansions to be finite, we present a truncation rule for the expansions. Finally, we derive a bound on the overall approximation error in terms of the assumed space/frequency concentration.
\end{abstract}



\end{frontmatter}

\section{Introduction}
Representing and processing three-dimensional volumetric data are central tasks in many applications, in particular, in medical and biological imaging~\cite{Frank2006,Natterer}. The efficiency and accuracy of algorithms for three-dimensional volumetric data processing crucially rely on the basis used to represent the data. In many applications, a natural assumption is that the underlying volume is (essentially) bandlimited, while  obviously being also space limited. In such a case, a natural basis for representing and processing the volumetric data is the so called ``generalized prolate spheroidal wavefunctions'' (GPSWF or PSWF)~\cite{slepian1964prolate}. The theory of PSWF has been derived in a seminal series of papers by Slepian et al.~\cite{slepian1961prolate,landau1961prolate,landau1962prolate,slepian1964prolate,slepian1978prolate}. The numerical algorithms for evaluating the PSWF in the one-dimensional case have been developed in~\cite{xiao2001prolate}, in the two-dimensional case in~\cite{shkolnisky2007prolate}, and recently in the three-dimensional case in~\cite{Lederman,BallPSWF2018}. Classical as well as recent results related to PSWF
can be found in~\cite{osipov2013prolate}.

In applications, the processed volumes are often specified by their samples on the Cartesian grid. In this note we derive representation and approximation schemes for bandlimited three-dimensional functions concentrated in a ball in three-dimensional space, which are specified by their samples on the Cartesian grid. This work is an extension to three dimensions of~\cite{landa2017approximation}, which considered the representation and approximation of two-dimensional sampled functions (images). Similarly to~\cite{landa2017approximation}, here we derive a method for expanding a three-dimensional function, specified by its samples, into a series of PSWF, derive a bound on the approximation error, present a truncation criterion for the expansion, and show that these results are also applicable in the case of ``almost'' bandlimited functions.

\section{Setting and mathematical preliminaries}

For a function $f:\mathbb{R}^3 \rightarrow \mathbb{C}$ such that $f \in L^2(\mathbb{R}^3)$, we define its Fourier transform as
\begin{equation}\label{eq:ft}
\mathcal{F}[f](\omega) := \int_{\mathbb{R}^{3}}f(x)e^{-\imath \ip{\omega}{x}}\, dx,\quad \omega \in \mathbb{R}^{3}.
\end{equation}
We say that $f$ is bandlimited if $\Omega:=\operatorname{supp}(\mathcal{F}(f))\subseteq \mathbb{R}^3 $ is bounded. Throughout this note, we denote by $R$ the unit ball in $\mathbb{R}^{3}$, and assume that $\Omega$ is a ball of radius $c$ centered at the origin, that is,
\begin{equation*}
R := \left \{x \in \mathbb{R}^{3},\ \norm{x}_{2}\leq 1\right \}, \qquad \Omega := c R,
\end{equation*}
for some $c>0$. We will henceforth assume that $f$ is $\Omega$-bandlimited, that is, can written as
\begin{equation}\label{eq:ift}
f(x) = \left ( \frac{1}{2\pi}\right )^{3} \int_{\Omega} g(\omega) e^{\imath \ip{\omega}{x}} \, d\omega = \left ( \frac{c}{2\pi}\right )^{3} \int_{R} g(c\omega) e^{\imath c \ip{\omega}{x}} \, d\omega,
\end{equation}
for $g\in L^{2}(\Omega)$. In such a case, we say that $f$ has bandlimit $c$.

The eigenfunctions of the operator on the right hand side of~\eqref{eq:ift} are called ``generalized prolate spheroidal wave functions'' (GPSWF), namely, they are the solutions to the equation
\begin{equation}\label{eq:prolateintegeq}
\alpha \psi(x) =\int_{R} e^{\imath c \ip{x}{y}}\psi(y)\, dy,\quad x\in R.
\end{equation}
In~\cite{slepian1964prolate} it was shown that the eigenvalue problem~\eqref{eq:prolateintegeq} has a countable set of eigenfunctions, which we denote by $\psi_{N,m,n}$, with a corresponding set of eigenvalues, denoted by  $\alpha_{N,n}$, where $n,N \in \mathbb{N}$ and $m\in \mathbb{Z}, -N\leq m \leq N$. Note that $\alpha_{N,n}$ is independent of $m$. The GPSWFs are orthogonal both on $R$ and $\mathbb{R}^3$, with respect to the standard inner products. Moreover, the GPSWFs form a complete system of $L^2(R)$ and of the subspace of bandlimited functions in $L^{2}(\mathbb{R}^3)$. We will assume that the GPSWF are normalized such that
\begin{equation*}
\norm{\psi_{N,m,n}(x)}_{R} = \sqrt{\int_{R} \abs{\psi_{N,m,n}(x)}^{2} \, dx}=1,
\end{equation*}
and so are orthonormal in the unit ball $R$.


In  \cite{slepian1964prolate} it was shown that the solutions $\psi(x)$ of~\eqref{eq:prolateintegeq} can be obtained by separation of variables in spherical coordinates $(r,\theta,\phi)$ as
\begin{equation}\label{eq:PSWF_KS_expansion}
\psi(r,\eta)= K(r)S_{m,N}(\theta,\phi),\ \theta \in [0,2\pi),\  \phi \in [0,\pi],
\end{equation}
where $K(r)$ is a univariate function to be defined shortly, and $S_{m,N}$ are the spherical harmonics defined by
\begin{equation*}
S_{m,N}(\theta,\phi)= \tilde{P}_N^m(\cos \theta )e^{\imath m\phi},\quad \theta \in [0,2\pi),\  \phi \in [0,\pi],
\end{equation*}
where $\tilde{P}_N^m$ is the normalized associated Legendre polynomial (see~\cite{press2007numerical}).
The functions $K(r)$ in~\eqref{eq:PSWF_KS_expansion} are shown in~\cite{slepian1964prolate} to be given as the solutions to the integral equation
\begin{equation}\label{eq:K_integral_equation}
\alpha K(r)=\int_{0}^{1} K(\rho) \, \rho^{2} \, H_{N}(cr\rho)\,d\rho,\quad N\in\mathbb{N},
\end{equation}
where $H_N(c r \rho)=\imath^N(2\pi)^{\frac{3}{2}}J_{N+\frac{1}{2}}(cr\rho)/\sqrt{c r \rho}$ and $J_{\nu}(x)$ are the Bessel functions of the first kind. Equation~\eqref{eq:K_integral_equation} has a countable set of solutions which we denote by $K_{n,N}(r)$, $n,N\in\mathbb{N}$, with corresponding eigenvalues $\alpha_{N,n}$. The eigenvalues of (\ref{eq:K_integral_equation}) coincide with those of (\ref{eq:prolateintegeq}). Note that the eigenvalues of~\eqref{eq:prolateintegeq} (and of~\eqref{eq:K_integral_equation}) depend only on the radial part of $\psi(r,\eta)$ (see (\ref{eq:PSWF_KS_expansion})) and therefore, don't depend on the index $m$. A numerical algorithm for evaluating the functions $K_{n,N}(r)$ in~\eqref{eq:K_integral_equation} has been recently described in~\cite{Lederman,BallPSWF2018}.

Since $\psi_{N,m,n}$ are complete for $\Omega$-bandlimited functions, any such function $f$ can be expanded as
\begin{equation}\label{eq:fexpan}
f(x) = \sum_{N,m,n} a_{N,m,n} \psi_{N,m,n}(x),\quad a_{N,m,n}=\int_{R} f(x) \psi_{N,m,n}(x) \, dx.
\end{equation}
However, in applications $f(x)$ is typically given only through its samples on the Cartesian grid, that is, we are only given the set
\begin{equation}\label{eq:fxk}
\left\{f(x_k)\, \Big | \,x_k=\frac{k}{L}\in Q\right\},
\end{equation}
where $Q=[-1,1]^3$ is the unit cube, $k\in\mathbb{Z}^{3}$ is a three-dimensional index vector, and $L$ is a positive integer known as the sampling rate. In the subsequent sections we show how to approximate the function $f(x)$ using only the samples~\eqref{eq:fxk}. Specifically, in Section~\ref{sec:sampling} we show how to approximate $f(x)$ for any $x\in R$, bound the approximation error, and extend the results to functions which are not strictly $\Omega$-bandlimited. Then, in Section~\ref{sec:numerical_results} we demonstrate numerically the theorems of Section~\ref{sec:sampling}. The results in Section~\ref{sec:sampling} are extensions of the results of~\cite{landa2017approximation} to three dimensions. Thus, the methodology used to derive the theorems in Section~\ref{sec:sampling} is similar to that used in~\cite{landa2017approximation}, but with two key differences -- the sampling theorem used in the proofs needs to be adapted to three dimensions as well as the bounds used therein.

\section{Sampling theorems for functions bandlimited to a ball}\label{sec:sampling}

Let $f:\mathbb{R}^3 \rightarrow \mathbb{R}^3$ be an $\Omega$-bandlimited function. Following~\eqref{eq:ft} and~\eqref{eq:ift}, we can write~$f$ as
\begin{equation*}
f(x)=(2\pi)^{-3}\int_{cR}\mathcal{F}[f](\omega)e^{\imath \ip{\omega}{x}}\,d\omega = \left(\frac{c}{2\pi}\right)^{3}\int_{R}\mathcal{F}[f](c \omega)e^{\imath c \ip{\omega}{x}} \, d\omega.
\end{equation*}
Since $F(\omega):=\left(\frac{c}{2\pi}\right)^3\mathcal{F}[f](cw)$ is supported on $R$, it can be expanded in GPSWFs as
\begin{equation}\label{eq:Fom}
F(\omega)=\sum_{N,m,n} b_{N,m,n} \psi_{N,m,n}(\omega),
\end{equation}
where the expansion coefficients $b_{N,m,n}$ are given by
\begin{equation}\label{eq:bNmn}
b_{N,m,n}=\int_R F(\omega)\overline{\psi_{N,m,n}}(\omega) \, d\omega = \frac{c^3}{(2\pi)^3}\int_{R}\mathcal{F}[f](c \omega)\overline{\psi_{N,m,n}}(\omega) \, d\omega.
\end{equation}
By using the inverse Fourier transform~\eqref{eq:ift}, it follows from \eqref{eq:prolateintegeq}, \eqref{eq:PSWF_KS_expansion} and \eqref{eq:Fom} that
\begin{equation}\label{eq:f_expansion_1}
f(x)=\int_R F(\omega)e^{\imath c \ip{\omega}{x}} \, d\omega = \sum_{N,m,n}b_{N,m,n}\alpha_{N,n}\psi_{N,m,n}(x),\quad  x\in R.
\end{equation}
Since the set $\left\{\psi_{N,m,n}\right\}_{N,m,n}$ is complete and orthogonal in the space of $\Omega$-bandlimited functions in $L^2(\mathbb{R}^3)$, by comparing~\eqref{eq:fexpan} and~\eqref{eq:f_expansion_1} we can write
\begin{equation*}
f(x)=\sum_{N,m,n}b_{N,m,n}\alpha_{N,n} \psi_{N,m,n}(x),\quad x\in \mathbb{R}^3.
\end{equation*}

To approximate the expansion coefficients $b_{N,m,n}$ of~\eqref{eq:bNmn} using only the samples~\eqref{eq:fxk} of~$f$, we use the besinc function~\cite{landa2017approximation,petersen1962sampling}, defined as the inverse Fourier transform of the indicator function $\chi_{\Omega}$, that is,
\begin{equation*}
h_c(x):= (2\pi)^{-3}\int_{\mathbb{R}^3} \chi_{\Omega}(y)e^{\imath \ip{x}{y}}dy = \left(\frac{c}{2\pi}\right)^3\int_R e^{\imath c \ip{x}{y}}dy,\quad x\in\mathbb{R}^{3}.
\end{equation*}
In essence, the besinc function is a generalization of the sinc function ($\operatorname{sinc}x=\sin x /x$) to higher dimensions. An explicit formula for the besinc function is given by the following lemma.

\begin{lemma}
\begin{equation}\label{eq:besinc}
h_c(x)=\frac{c^{\frac{3}{2}}}{\sqrt{2} \pi^{\frac{3}{2}}}\frac{J_{\frac{3}{2}}(c\|x\|)}{\|x\|^{\frac{3}{2}}}, \quad x\in R\setminus \left\{0\right\}.
\end{equation}
\end{lemma}

\begin{proof}
Let $x=\left(0,0,p\right)^{T}$ for $0<p<1$. Then, we have that
\begin{align*}
h_c(x) &=  \left(\frac{c}{2\pi}\right)^3 \int_{-1}^{1}dy_3 \int_{y_1^2+y_2^2\leq 1-y_3^2} e^{\imath c p y_3}dy_1dy_2  = \pi \left(\frac{c}{2\pi}\right)^3 \int_{-1}^{1} \left(1-y_3^2\right)e^{\imath c p y_3} dy_3.
\end{align*}
By applying the Poisson representation formula for Bessel functions~\cite{gradshteyn2014table}
\begin{equation*}
J_{\nu}(z)=\frac{\left(\frac{z}{2}\right)^{\nu}}{\sqrt{\pi}\Gamma(\nu +\frac{1}{2})}\int_{-1}^1 \left(1-s^2\right)^{\nu -\frac{1}{2}}e^{\imath z s}ds,
\end{equation*}
we get
\begin{equation*}
h_c(x)=\frac{c^{\frac{3}{2}}}{\sqrt{2}(\pi)^{\frac{3}{2}}} \frac{J_{\frac{3}{2}}(cp)}{p^{\frac{3}{2}}}=\frac{c^{\frac{3}{2}}}{\sqrt{2}(\pi)^{\frac{3}{2}}}\frac{J_{\frac{3}{2}}(c\|x\|)}{\|x\|^{\frac{3}{2}}}.
\end{equation*}
Since the Fourier transform of a radial function is a radial function (the same is true for the inverse transform), given a general vector $x\in R$ there exists $A\in O(3)$ such that $Ax=u:=\left(0,0,p\right)^{T}$ for some $0<p<1$. Since $\chi_{\Omega}$ is a radial function, we have that~\eqref{eq:besinc} is true for all $x\in R\setminus \left\{0\right\}$.

\end{proof}

The next lemma is an auxiliary lemma used later to approximate the expansion coefficients~$b_{N,m,n}$ of~\eqref{eq:bNmn}.

\begin{lemma}\label{thm:sum_dk}
Let $d_k \in \mathbb{R}$ be arbitrary numbers. Define
\begin{equation*}
\tilde{b}_{N,m,n} := \frac{c^3}{(2\pi)^3} \sum_{\frac{k}{L}\in R} d_k\overline{\left(\alpha_{N,n}\psi_{N,m,n}\left(\frac{k}{L}\right)\right)}.
\end{equation*}
Then,
\begin{equation}
\sum_{\frac{k}{L}\in R}d_kh_{c,k}(x)=\sum_{N,m,n}\alpha_{N,n}\tilde{b}_{N,m,n}\psi_{N,m,n}(x),\quad x\in \mathbb{R}^3,
\end{equation}
where $\alpha_{N,n}$ and $\psi_{N,m,n}$ are the eigenvalues and eigenfunctions of~\eqref{eq:prolateintegeq}, and
\begin{equation}\label{Shiftbesinc}
h_{c,k}(x):=h_c(x-\frac{k}{L}).
\end{equation}
\end{lemma}
The proof of Lemma~\ref{thm:sum_dk} is a straightforward generalization of Lemma~1 in~\cite{landa2017approximation} and is therefore omitted.

The following lemma bounds the error when approximating an $\Omega$-bandlimited function $f$ by a series of GPSWFs, where the  expansion coefficients are computed using only the samples~\eqref{eq:fxk}.

\begin{lemma}
Let $f \in L^2(\mathbb{R}^3)$ and $\Omega$-bandlimited, where $\Omega=cR$, and suppose that $c\leq \pi L$. Define
\begin{equation}\label{eq:bhat}
\hat{b}_{N,m,n}:=\frac{c^3}{(2\pi L)^3} \sum_{\frac{k}{L}\in R} f\left(\frac{k}{L}\right)\overline{\left(\alpha_{N,n}\psi_{N,m,n}\left(\frac{k}{L}\right)\right)},
\end{equation}
and an approximation of $f$ in the unit ball by
\begin{equation}\label{eq:fhat}
\hat{f}(x):=\sum_{N,m,n} \hat{a}_{N,m,n}\psi_{N,m,n}(x),\quad \hat{a}_{N,m,n}:=\alpha_{N,n}\hat{b}_{N,m,n}.
\end{equation}
Then,
\begin{equation}\label{eq:fullapprox}
\|f-\hat{f}\|_{L^2(R)}\leq \frac{1}{L^3}\sqrt{\sum_{\frac{k}{L}\notin R}\left|f\left(\frac{k}{L}\right)\right|^2}\left|\left|\sqrt{\sum_{\frac{k}{L}\notin R}\left|h_{c,k}\left(x\right)\right|^2}\right|\right|_{L^2(R)}.
\end{equation}
\end{lemma}
The proof of Lemma~\ref{thm:sum_dk} is a straightforward generalization of Theorem~1 in~\cite{landa2017approximation} and is therefore omitted.

For the bound in~\eqref{eq:fullapprox} to be of practical use, we need to show that the rightmost term in~\eqref{eq:fullapprox} is small.
\begin{lemma}
Define
\begin{equation}\label{eq:ksic}
\xi_c(x):=\sqrt{\sum_{\frac{k}{L}\notin R}\left|h_{c,k}\left(x\right)\right|^2}.
\end{equation}
Then,
\begin{equation}\label{eq:ksic_bound}
\left|\left|\xi_c\right|\right|_{L^2(R)}:=\eta\leq \frac{4\pi}{3}c^{\frac{3}{2}}L^{\frac{3}{2}}.
\end{equation}
\end{lemma}
\begin{proof}
We have
\begin{eqnarray}
h_c\left(x-\frac{k}{L}\right)&=&h_c\left(\frac{k}{L}-x\right)= \left(\frac{c}{2\pi}\right)^3\int_R e^{\imath c \ip{\frac{k}{L}-x}{y}}dy \nonumber\\
&=&\left(\frac{c}{2\pi}\right)^3 \int_{\frac{c}{L}R} \left(\frac{L}{c}\right)^3 e^{\imath \ip{k}{y}} e^{-\imath L \ip{x}{y}}dy \nonumber \\
&=&\left(\frac{1}{2\pi}\right)^3 \int_{\left[-\pi,\pi\right]^3}L^3\chi_{\frac{c}{L}R}(y)e^{-\imath L \ip{x}{y}} e^{\imath \ip{k}{y}}dy \nonumber \\
&=& \hat{\zeta^x}_{-k} ,\nonumber
\end{eqnarray}
where
\begin{equation}
\zeta^x(y):= L^3\chi_{\frac{c}{L}R}(y)e^{-\imath L \ip{x}{y}}, \nonumber
\end{equation}
and $\hat{\zeta^x}_{-k}$ denotes the Fourier coefficient of $\zeta^x$ which corresponds to $-k\in \mathbb{Z}^3$ (here we treat $x$ as a constant). Since $c\leq \pi L$, we have that $\frac{c}{L}R \subseteq	\left[-\pi,\pi\right]^3$. Therefore, by Bessel's inequality
\begin{eqnarray}
\xi_c^{2}(x) = \sum_{\frac{k}{L}\notin R}\left|h_{c,k}\left(x\right)\right|^2 &=& \sum_{k\in \mathbb{Z}^3:\frac{k}{L}\notin R}\left|\hat{\zeta^x}_{-k}\right|^2 \nonumber \\
&\leq& \left| \left| \zeta^x \right| \right|_{L^2\left(\left[-\pi,\pi\right]^3\right)}^2 = L^6\operatorname{vol}\left(\frac{c}{L}R\right)=\frac{4\pi}{3}c^3L^3. \nonumber
\end{eqnarray}
This implies that, pointwise in $R$, $\xi_c(x)\leq \sqrt{\frac{4\pi}{3}}c^{\frac{3}{2}}L^{\frac{3}{2}}$, which implies \eqref{eq:ksic_bound}.
\end{proof}
We provide a more in depth analysis of the behavior of $\xi_c$ in~\ref{sec:asymptotics}. Specifically, we demonstrate that for $r_1<1$
\begin{equation*}
\frac{1}{c^6}\left|\left| \xi_c \right| \right|_{L^2(r_1R)}^2 = O(\frac{1}{L}).
\end{equation*}
This asymptotic relation holds even for relatively small values of $L$, as can be seen in Figure~\ref{fig:Xi_c}.

\begin{figure}
\begin{center}
\begin{tabular}{ccccclllll}
\includegraphics[width=110mm]{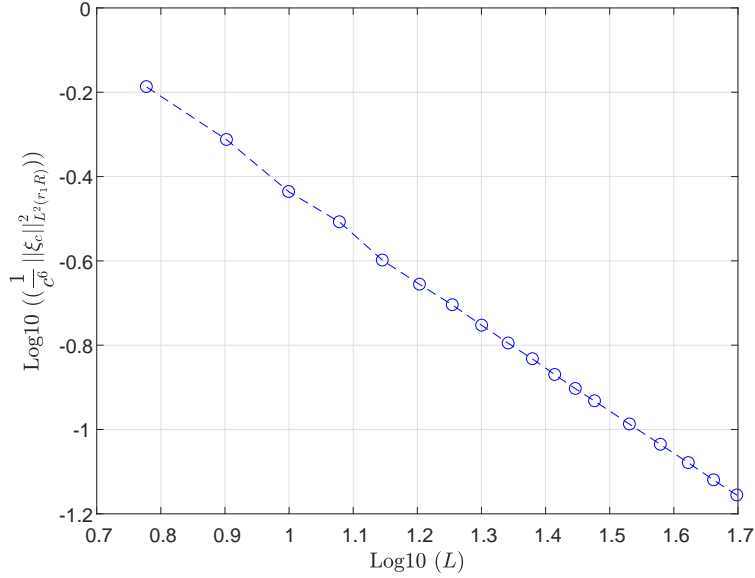}
\end{tabular}
\end{center}
\caption{Computed values of $\log_{10}(\frac{1}{c^6}\left|\left| \xi_c \right| \right|_{L^2(r_1R)}^2)$ for $L\leq	50$, $c=\pi L$, $r_1=0.95$. The slope of the linear fit is $\approx -1.05038$. }
\label{fig:Xi_c}
\end{figure}

For digital implementations, the infinite series in~\eqref{eq:fhat} must be truncated. The following theorem bounds the approximation error induced by such a truncation.
\begin{theorem}\label{thm:truncation}
Suppose that $f\in L^2(\mathbb{R}^3)$ is an $\Omega$-bandlimited function with $\left|\left|f\chi_{R^c}\right|\right|_{L^2(\mathbb{R}^3)}\leq \epsilon$. Then, for every finite set of indices $\Pi$,
\begin{equation}\label{eq:truncation}
\left|\left|f-\sum_{(N,m,n)\in \Pi}b_{N,m,n}\alpha_{N,n}\psi_{N,m,n}\right|\right|_{L^2(R)}\leq \epsilon \sqrt{\max_{(N,m,n)\notin \Pi}\left[\frac{\tilde{\alpha}_{N,n}}{1-\tilde{\alpha}_{N,n}}\right]},
\end{equation}
where $b_{N,m,n}$ is given by~\eqref{eq:bNmn}, $\tilde{\alpha}_{N,n}=\left(\frac{c}{2\pi}\right)^3\left|\alpha_{N,n}\right|^2$ and $\alpha_{N,n}$ is the eigenvalue corresponding to $\psi_{N,m,n}$.
\end{theorem}

Theorem~\ref{thm:truncation} above is the three-dimensional counterpart of Theorem~3 in~\cite{landa2017approximation}. As the proof of the latter is independent of the dimension of the problem, we omit the proof of Theorem~\ref{thm:truncation}.

In light of Theorem~\ref{thm:truncation}, for an $\Omega$-bandlimited function $f\in L^2(\mathbb{R}^3)$ and a set of indices $\Pi$, the approximation error $\left|\left|f-\sum_{(N,m,n) \in \Pi}\hat{a}_{N,m,n}\psi_{N,m,n}\right|\right|_{L^2(R)}$ with $\hat{a}_{N,m,n}$ given in~\eqref{eq:fhat} is given by
\begin{align}
\left|\left|f-\sum_{(N,m,n)\in \Pi}\hat{a}_{N,m,n}\psi_{N,m,n}\right|\right|_{L^2(R)} &\leq  \left|\left|f-\sum_{(N,m,n)\in \Pi}a_{N,m,n}\psi_{N,m,n}\right|\right|_{L^2(R)} \nonumber\\
&+ \left|\left|\sum_{(N,m,n)\in \Pi}(a_{N,m,n}-\hat{a}_{N,m,n})\psi_{N,m,n}\right|\right|_{L^2(R)}, \label{eq:aahat}
\end{align}
where $a_{N,m,n} = \alpha_{N,n}b_{N,m,n}$. The term~\eqref{eq:aahat} satisfies
\begin{align*}
\left|\left|\sum_{(N,m,n)\in \Pi}(a_{N,m,n}-\hat{a}_{N,m,n})\psi_{N,m,n}\right|\right|_{L^2(R)} &\leq \left|\left|\sum_{(N,m,n)}(a_{N,m,n}-\hat{a}_{N,m,n})\psi_{N,m,n}\right|\right|_{L^2(R)} \\
&= \left|\left|f-\hat{f} \right|\right|_{L^2(R)} .
\end{align*}
By combining~\eqref{eq:fullapprox} and~\eqref{eq:truncation}, we get that
\begin{equation*}
\left|\left|f-\sum_{(N,m,n)\in \Pi}\hat{a}_{N,m,n}\psi_{N,m,n}\right|\right|_{L^2(R)} \leq \frac{\eta}{L^3}\sqrt{\sum_{\frac{k}{L}\notin R}\left|f\left(\frac{k}{L}\right)\right|^2}  + \epsilon \sqrt{\max_{(N,m,n)\notin \Pi}\left[\frac{\tilde{\alpha}_{N,n}}{1-\tilde{\alpha}_{N,n}}\right]},
\end{equation*}
where $\eta$ is defined in \eqref{eq:ksic_bound}.

In order to address the approximation of non-bandlimited functions, we define the energy of a function $f$ outside of $\Omega$ by
\begin{equation}\label{eq:deltac}
\delta_c := \frac{1}{(2\pi)^{\frac{3}{2}}}\left|\left|\mathcal{F}[f]\right|\right|_{L^2(\Omega^c)}.
\end{equation}
The next theorem gives an error estimate for the case of a non-bandlimited function, under additional assumptions on the samples of~$f$.
\begin{theorem}
Suppose that $f\in L^2(\mathbb{R}^3)$ and $\left\{f\left(\frac{k}{L}\right)\right\}_{k\in \mathbb{Z}^3}\in l^2$. Define the coefficients $\hat{b}_{N,m,n}$ and the approximating function $\hat{f}(x)$ as in~\eqref{eq:bhat} and~\eqref{eq:fhat}, respectively. If $c\leq \pi L$, then,
\begin{equation}\label{eq:approxnonb}
\|f-\hat{f}\|_{L^2(R)}\leq \frac{\eta}{L^3}  \sqrt{\sum_{\frac{k}{L}\notin R}\left|f\left(\frac{k}{L}\right)\right|^2}+2\delta_c.
\end{equation}
\end{theorem}

Note that the bound in~\eqref{eq:approxnonb} is different from the bound in Theorem 5 in~\cite{landa2017approximation}, due the change from $\mathbb{R}^{2}$ to $\mathbb{R}^3$.

An error estimate for the approximation of a non-bandlimited function by a truncated series of GPSWFs (analogous to Theorem~\ref{thm:truncation}) is given in the following theorem.
\begin{theorem}\label{thm:truncated_almost_BL}
Suppose that $f\in L^2(\mathbb{R}^3)$ with $\left|\left|f\chi_{R^c}\right|\right|_{L^2(\mathbb{R}^3)}\leq \epsilon$ and $\left\{f\left(\frac{k}{L}\right)\right\}_{k\in \mathbb{Z}^3}\in l^2$. Then, for every finite set of indices $\Pi$
\begin{equation}\label{eq:truncnonb}
\left|\left|f-\sum_{(N,m,n)\in \Pi}b_{N,m,n}\alpha_{N,n}\psi_{N,m,n}\right|\right|_{L^2(R)}\leq (\epsilon+\delta_c) \sqrt{\max_{(N,m,n)\notin \Pi}\left[\frac{\tilde{\alpha}_{N,n}}{1-\tilde{\alpha}_{N,n}}\right]}+2\delta_c,
\end{equation}
where $\delta_{c}$ is given by~\eqref{eq:deltac}, $\tilde{\alpha}_{N,n}=\left(\frac{c}{2\pi}\right)^2\left|\alpha_{N,n}\right|^2$ and $\alpha_{N,n}$ is the eigenvalue corresponding to $\psi_{N,m,n}$.
\end{theorem}
To simplify the bounds in the theorems above (e.g. Theorem~\ref{thm:truncated_almost_BL}), we define a ``truncation parameter'' $T>0$ and a corresponding set of indices
\begin{equation}\label{eq:PiT}
\Pi_T : = \left\{(N,m,n) \ : \ \sqrt{\frac{\tilde{\alpha}_{N,n}}{1-\tilde{\alpha}_{N,n}}}>T, \ -N \leq m \leq N  \right\}.
\end{equation}
Then, by combining~\eqref{eq:approxnonb} and~\eqref{eq:truncnonb} we obtain the simplified error estimate
\begin{equation}\label{FError}
\left|\left|f-\sum_{(N,m,n) \in \Pi_T}\hat{a}_{N,m,n}\psi_{N,m,n}\right|\right|_{L^2(R)} \leq (\epsilon +\delta_c)T +\frac{1}{L^3}\eta \sqrt{\sum_{\frac{k}{L}\notin R}\left|f\left(\frac{k}{L}\right)\right|^2}+4\delta_c.
\end{equation}
Note that the approximation error given by the right hand side of~\eqref{FError} is governed by two factors. The first is $c$, which arises from truncating $f$ in the Fourier domain; the second is $T$, which dictates the number of basis functions used in the approximation.

As in \cite{landa2017approximation}, the dependence of $\Pi_{T}$ on $T$ is of interest. An analysis carried out in~\cite{greengard2018generalized} implies that the number of tuples in $\Pi_{T}$ is given by
\begin{equation*}
|\Pi_T| = \frac{c^3}{32} -\frac{1}{2\pi^2} c^2 \log(c) \log(T) + o(c^2 \log(c)).
\end{equation*}
Table~\ref{tbl:Ngpwfs} presents the ratio between $|\Pi_T|$ (the number of GPSWFs used to expand a function) and the number of samples in the unit ball, for various values of $T$ and $L$.

\begin{landscape}
\begin{table}
\centering
\begin{tabular}{|c|c|c|c|c|c|c|c|c|c|c|}
\hline
Samples & 17071 & 33371 &57747 & 91911 &137059 & 195167 &267731 &356559 &462751 &588739 \\ \hline \hline
\backslashbox{$\log_{10}T$}{$L$} & 16 &20 &24 &28 &32 &36& 40&44 &48 & 52 \\ \hline
-6 & 1.65 &1.41 &1.27 &1.16	&1.08 &1.02 & 0.96 &	0.89 &0.83 &0.72
 \\ \hline
-5 & 1.45 &1.26 &1.14 &1.05 &0.99 &0.94 & 0.89 & 0.84 &0.78 & 0.69 \\ \hline
-4 & 1.27 &1.12 &1.02 &0.95 &0.9 &0.86 & 0.83 &	0.79 &0.74 &0.66  \\ \hline
-3 & 1.08 &0.97 &0.90 &0.85 &0.81 &0.78 & 0.76 &	0.73 &0.69 &0.62  \\ \hline
-2 & 0.90 &0.83 &0.78 &0.74 &0.72 &0.70 & 0.68 &	0.67 &0.64 &0.58 \\ \hline
-1 & 0.71 &0.68 &0.66 &0.64 &0.63 &0.62 & 0.61 &0.60 &0.59 &0.54 \\ \hline
0 & 0.53 &0.52 &0.52 &0.52 &0.52 &0.52 & 0.52 &0.52 &0.52 &0.51 \\ \hline
1 & 0.37 &0.39 &0.41 &0.42 &0.43 &0.44 &0.44 &0.45 &0.46 &0.45  \\ \hline
2 & 0.27 &0.31 &0.33 &0.35 &0.367 &0.38 &0.39 &0.40 &0.41 &0.41 \\ \hline
3 & 0.21 &0.24 &0.27 &0.29 &0.32 &0.33 &0.34 &0.36 &0.37 &0.38 \\ \hline
4 & 0.16 &0.20 &0.23 &0.25 &0.27 &0.29 &0.31 &0.32 &0.33 &0.23 \\ \hline
5 & 0.12 &0.16 &0.19 &0.21 &0.22 &0.25 &0.27 &0.26 &0.27 &0.27 \\ \hline
6 & 0.09 &0.12 &0.16 &0.15 &0.20 &0.18 &0.19 &0.19 &0.20 &0.20 \\ \hline
\end{tabular}
\caption{Ratio between the number of GPSWFs required to expand a function and the number of samples in the unit ball, for various values of $T$ and $L$.}
\label{tbl:Ngpwfs}
\end{table}
\end{landscape}

Another important property of GPSWFs is that the vectors obtained by sampling them on a Cartesian grid are ``almost orthogonal''. This property is discussed in Appendix B.

\section{Numerical Results}\label{sec:numerical_results}

In this section we demonstrate numerically the approximation theorems of Section~\ref{sec:sampling}. The numerical evaluation of the functions $\psi_{N,m,n} $ (the solutions of~\eqref{eq:prolateintegeq}) is based on their separation of variables~\eqref{eq:PSWF_KS_expansion}, where the radial part is evaluated using the algorithm in~\cite{Lederman}, and the spherical harmonics are evaluated as explained in Section 6.7 in~\cite{press2007numerical}. All algorithms have been implemented in $\textsc{MATLAB}^{TM}$, and are available at \url{http://www.math.tau.ac.il/~yoelsh/}.

To demonstrate our approximation scheme, we apply it to Gaussians of the form
\begin{equation}\label{Gaussian}
f(x):=(2\pi \sigma)^{-\frac{3}{2}}e^{-\frac{|x-\mu|^2}{2\sigma}}, \quad x\in R.
\end{equation}
The parameter $\mu$ shifts the center of the Gaussian from the origin so that not only GPSWFs of order zero are used in the expansions. The three-dimensional Fourier transform of $f$ is given by
\begin{equation}\label{FourierGaussian}
\mathcal{F}[f](\omega) = e^{-\imath \ip{\omega}{\mu}} e^{-\frac{|\omega|^2\sigma}{2}}.
\end{equation}
Equations~\eqref{Gaussian} and~\eqref{FourierGaussian} imply that the error in the approximation scheme depends on the interplay between $\sigma$ and $L$  (we set $c=\pi L$).

We demonstrate the results for $\mu=(0.1,0.1,0.1)^T$  and various values of $T$, $\sigma$, and $L$, by evaluating  both sides of~\eqref{FError}. The right hand side is evaluated numerically, using a quadrature formula for the unit ball. On the left hand side, $\eta$ is estimated with the bound in (\ref{eq:ksic_bound}), and the term
\begin{equation*}
\sqrt{\sum_{\frac{k}{L}\notin R}\left|f\left(\frac{k}{L}\right)\right|^2}
\end{equation*}
in~\eqref{FError} is estimated using a sufficient number of samples of $f$ outside the unit ball. The parameters $\epsilon$ and $\delta_c$ are evaluated analytically using the properties of Gaussians. The results are shown in Figure~\ref{fig:ErrBound1} and Figure~\ref{fig:ErrBound2}. These figures show that when $\sigma$ is large (the concentration in space is low), $\epsilon$ (see e.g. Theorem~\ref{thm:truncated_almost_BL}) dominates the error. On the other hand, whenever $\sigma$ is small, the ``energy'' of $f$ in the Fourier domain decays more slowly, which leads to $\delta_{c}$ (see~\eqref{eq:deltac}) being the dominating term in the error. The smallest approximation error is achieved when $\epsilon$ and $\delta_{c}$ are approximately equal.

\begin{figure}
\begin{center}
\subfloat[]{
\includegraphics[width=0.5\textwidth]{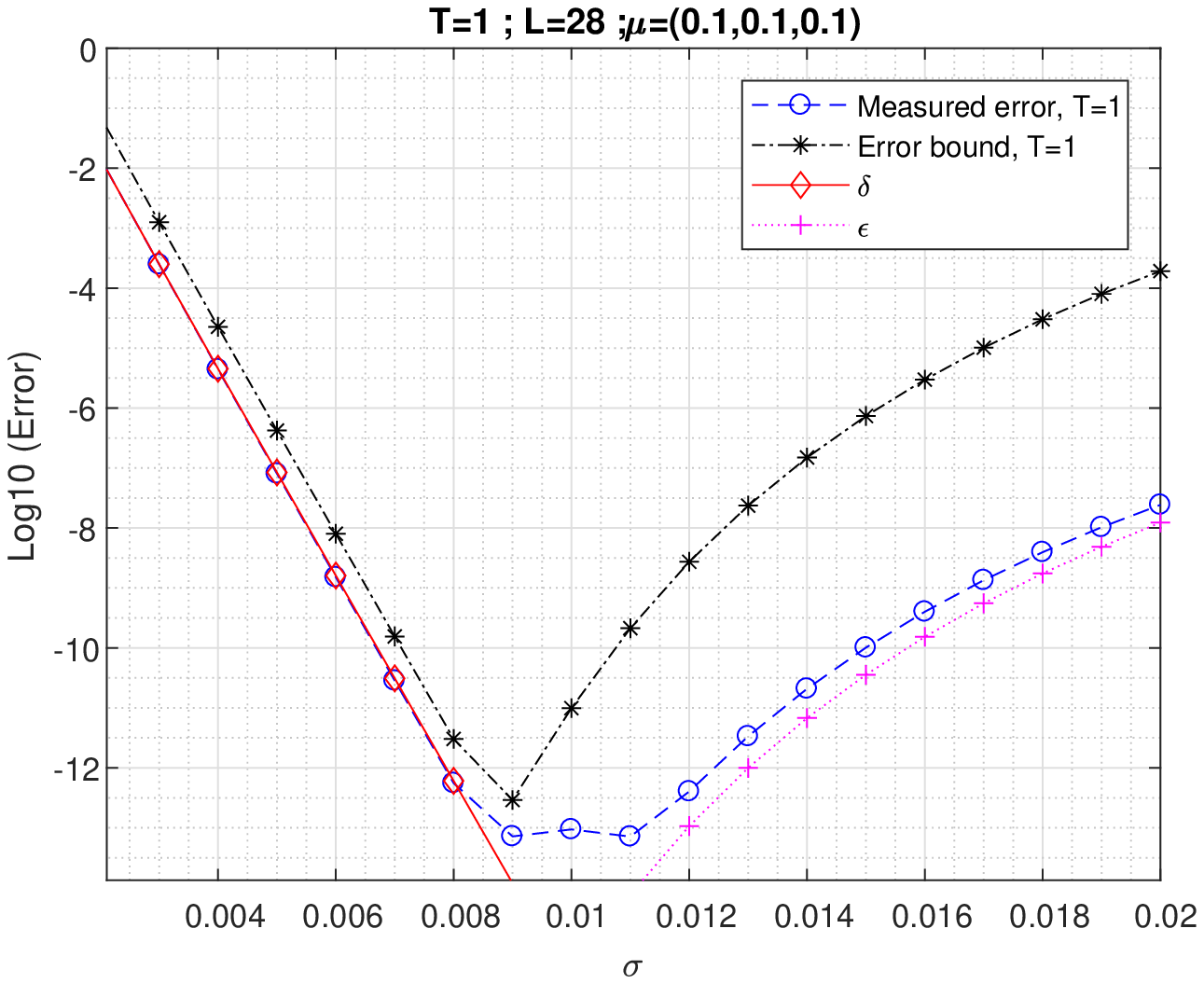}
\label{fig:T1L28}
}
\subfloat[]{\includegraphics[width=0.5\textwidth]{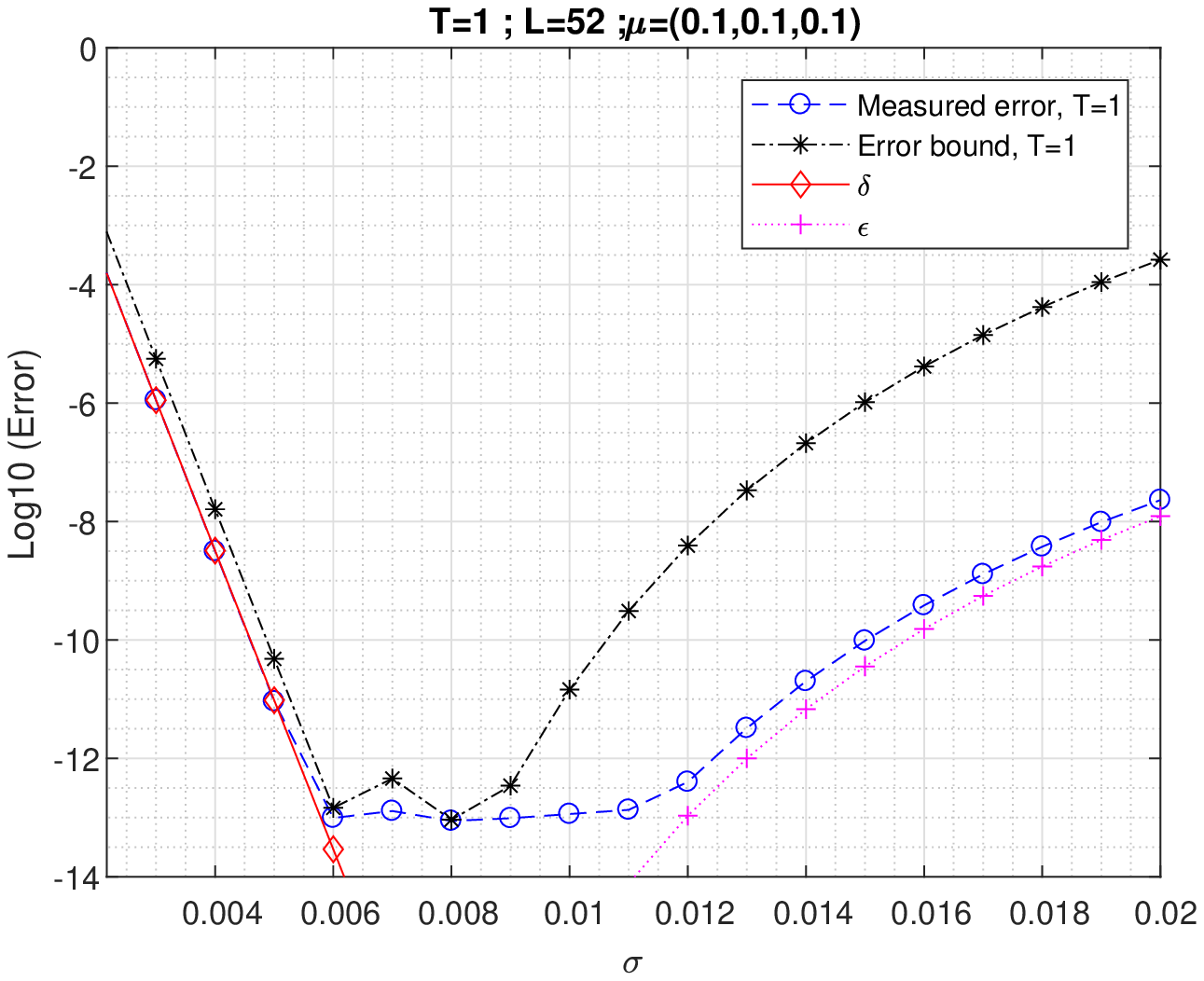}
\label{fig:T1L52}
}
\end{center}
\caption{Measured approximation error versus the estimated error bound for $T=1$ , \protect\subref{fig:T1L28}~$L=28$ and \protect\subref{fig:T1L52}~$L=52$. }
\label{fig:ErrBound1}
\end{figure}

\begin{figure}
\begin{center}
\subfloat[]{
\includegraphics[width=0.5\textwidth]{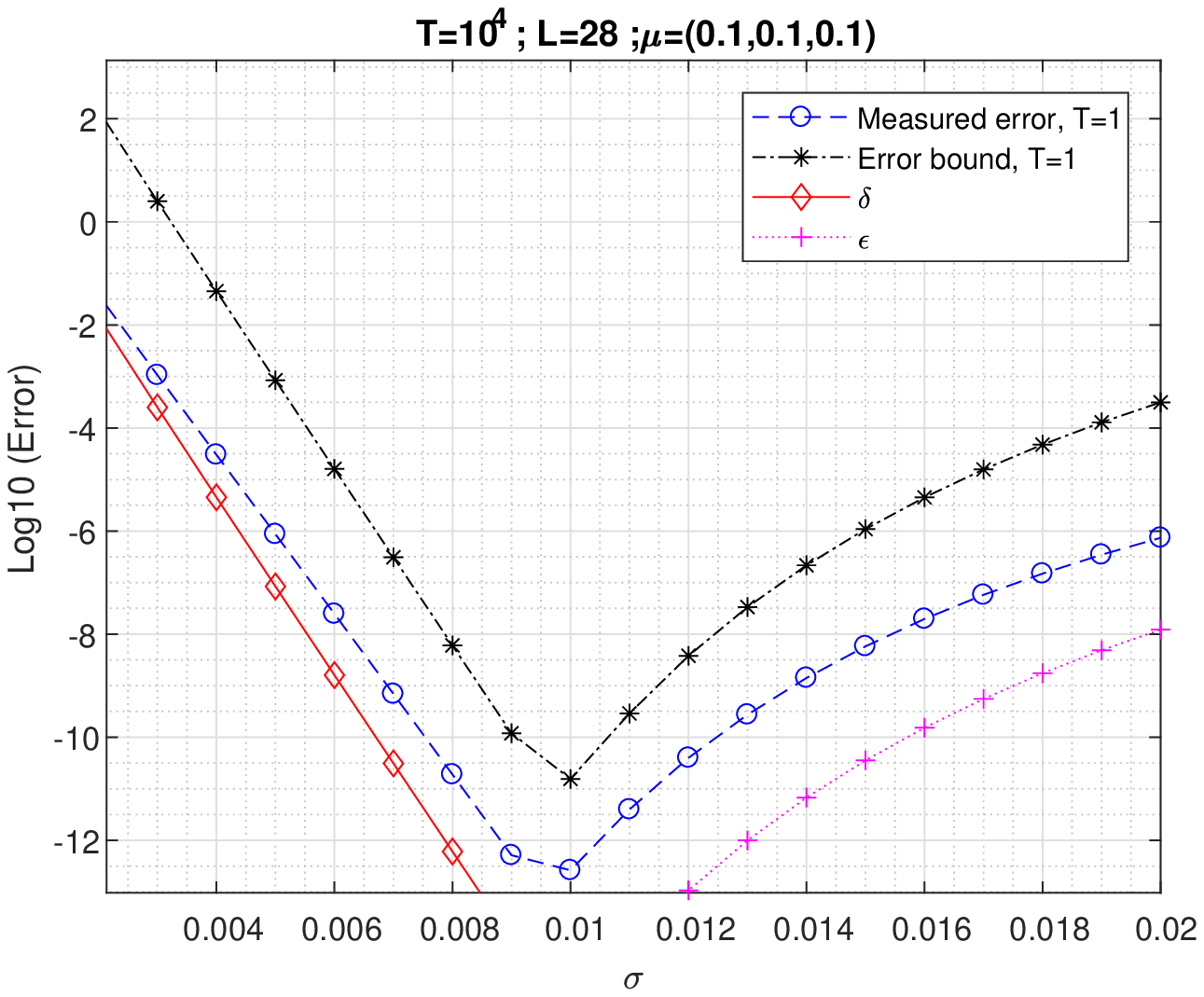}
\label{fig:T106L28}
}
\subfloat[]{
\includegraphics[width=0.5\textwidth]{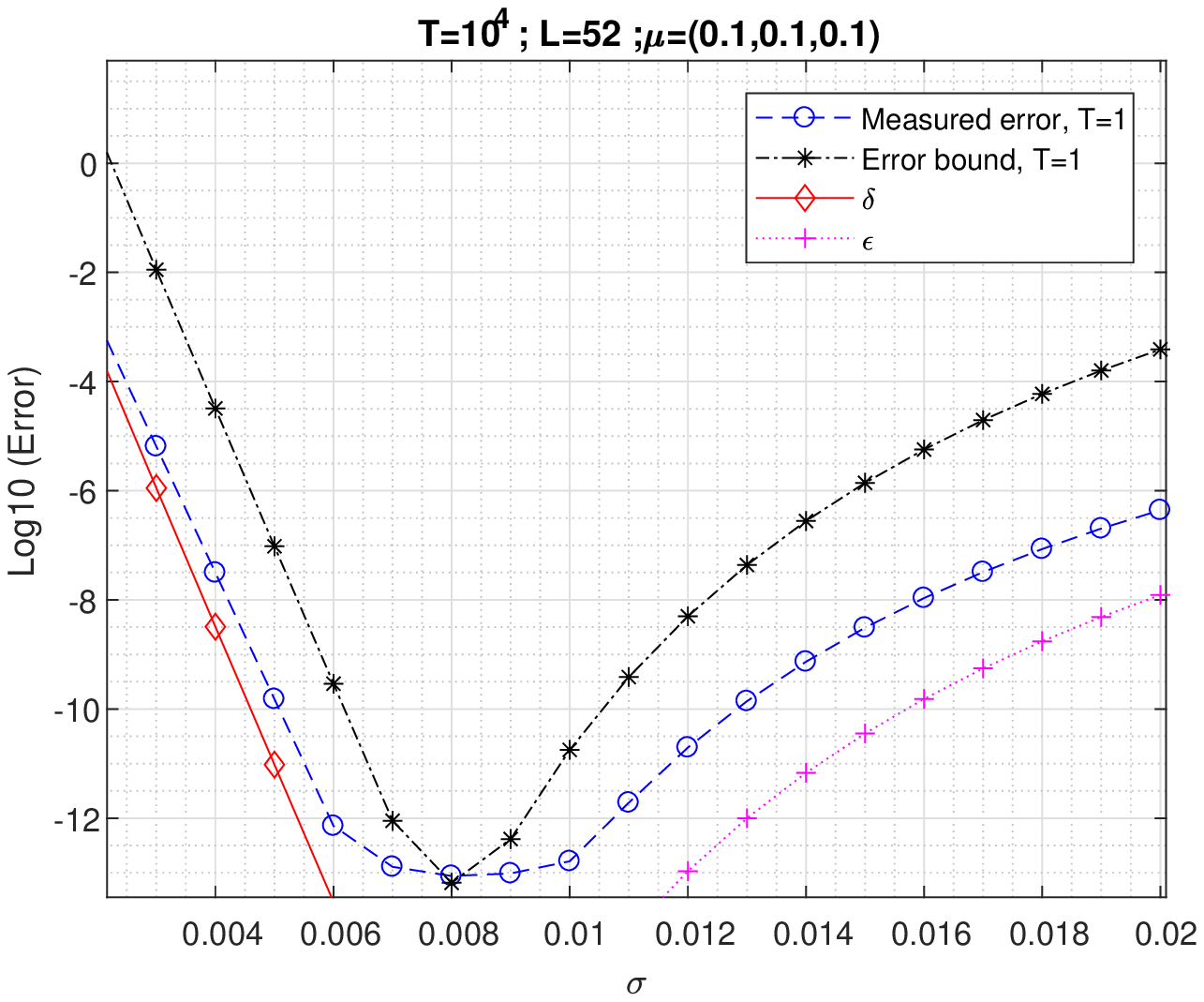}
\label{fig:T106L52}
}
\end{center}
\caption{Measured approximation error versus the estimated error bound for $T=10^4$ , \protect\subref{fig:T106L28}~$L=28$ and \protect\subref{fig:T106L52}~$L=52$.}
\label{fig:ErrBound2}
\end{figure}

\section{Summary}
\label{summary}
In this work, we have extended the GPSWFs-based approximation scheme presented in~\cite{landa2017approximation} to functions on $\mathbb{R}^3$, which are sufficiently concentrated in space and frequency. The approximation scheme is based on sampling the approximated function on a Cartesian grid and requires only discrete scalar products. We have also presented error bounds for the approximation error, and demonstrated them numerically.

\appendix

\section{Asymptotic behavior of $\xi_c$}\label{sec:asymptotics}
We would like to derive a bound for $\left|\left| \xi_c \right| \right|_{L^2(r_1R)}^2 $ for $r_1 < 1$ (see~\eqref{eq:ksic} for the definition of $\xi_c$). We denote by $B(0,r)\subset \mathbb{R}^3$ the ball of radius $r$ centered at zero, and by $S(0,r)$ the boundary of $B(0,r)$.  The Bessel function of the first kind corresponding to order $\nu = \frac{3}{2}$ is given by (see \cite{gradshteyn2014table}) \begin{equation}\label{eq:J32}
J_{\frac{3}{2}}(z) = \sqrt{\frac{2}{\pi}} \frac{1}{z^{\frac{3}{2}}} \left(\sin(z)-z\cos(z)\right).
\end{equation}
Substituting~\eqref{eq:J32} into \eqref{eq:besinc}, we find that
\begin{equation*}
h_c(x) = \frac{1}{\pi^2} \frac{1}{\left|\left|x\right|\right|^3} \left(\sin(c\left|\left|x\right|\right|) - c\left|\left|x\right|\right|\cos(c\left|\left|x\right|\right|) \right),
\end{equation*}
from which we obtain the estimate
\begin{equation*}
\left|h_c(x)\right|^2 \leq \frac{1}{\pi^4} \frac{\left(1+c\left|\left|x\right|\right|\right)^2}{\left|\left|x\right|\right|^6}.
\end{equation*}
Therefore,
\begin{equation}\label{eq:ksiestimate}
\left|\xi_c(x)\right|^2 := \sum_{\frac{k}{L}\notin R}\left|h_{c}\left(x-\frac{k}{L}\right)\right|^2 \leq \sum_{\frac{k}{L}\notin R} \frac{1}{\pi^4} \frac{\left(1+cr_1+cL^{-1}\left|\left|k\right|\right|\right)^2                                                                                                                                                                                                                                                                                                                                                                                                                                                                                                                                                                                                                                                                                                 }{\left(L^{-1}\left|\left|k\right|\right|-r_1\right)^6}, \quad \left|\left|x\right|\right|<r_1.
\end{equation}
We would like to bound the series on the right hand side of~\eqref{eq:ksiestimate}. For $k\in \mathbb{Z}^3$, we define the cube $Q_k = L^{-1}\left([k_x,k_x+1)\times [k_y,k_y+1)\times [k_z,k_z+1)\right)$, with $\operatorname{vol}(Q_k)=\frac{1}{L^3}$. Let $h:(r_1,\infty) \rightarrow \mathbb{R}$ be defined by
\begin{equation}\label{eq:hz}
h(z) : = \frac{\left(1+cr_1+cz\right)^2                                                                                                                                                                                                                                                                                                                                                                                                                                                                                                                                                                                                                                                                                                 }{\left(z-r_1\right)^6}.
\end{equation}
Then, the right hand side of \eqref{eq:ksiestimate} is equal to $\frac{L^3}{\pi^4} \sum_{\frac{k}{L} \notin R}\frac{1}{L^3} h(L^{-1}\left|\left|k\right|\right|)$, which is a Riemann sum multiplied by $\frac{L^3}{\pi^4}$. It can be easily verified that $h$ is monotonically decreasing. Therefore, for any $p \in Q_k$,
\begin{equation*}
h(L^{-1}\left|\left|k+(1,1,1)^T\right|\right|)\leq h(\left|\left|p\right|\right|) \leq h(L^{-1}\left|\left|k\right|\right|),
\end{equation*}
which gives the estimate
\begin{equation*}
\frac{1}{L^3}h(L^{-1}\left|\left|k+(1,1,1)^T\right|\right|)\leq \int_{Q_k} h(\left|\left|p\right|\right|) dp \leq \frac{1}{L^3}h(L^{-1}\left|\left|k\right|\right|).
\end{equation*}
Combining this estimate with \eqref{eq:ksiestimate} we obtain that
\begin{equation}\label{eq:xic bound}
\left|\xi_c(x)\right|^2 \leq \frac{L^3}{\pi^4} \int_{\mathbb{R}^3\setminus B(0,1-\frac{\sqrt{3}}{L})} h(\left|\left|p\right|\right|) dp,
\end{equation}
for $\left|\left|x\right|\right|<r_1<1-\frac{\sqrt{3}}{L}$ . The integral on the right hand side of~\eqref{eq:xic bound} can be simplified as
\begin{equation*}
\int_{\mathbb{R}^3\setminus B(0,1-\frac{\sqrt{3}}{L})} h(\left|\left|p\right|\right|) dp = \int_{1-\frac{\sqrt{3}}{L}}^{\infty} \left(\int_{S(0,t)}h(\left|\left|p\right|\right|)dS(p)\right) dt = 4\pi \int_{1-\frac{\sqrt{3}}{L}}^{\infty} t^2h(t) dt.
\end{equation*}
Thus, assuming that $L \ge 8$ and $c=\pi L$, we get
\begin{equation}\label{eq:KsiPointEst}
\left|\xi_c(x)\right|^2 \leq \frac{4L^3}{\pi^3}\int_{1-\frac{\sqrt{3}}{L}}^{\infty} t^2h(t) dt \leq \frac{4L^3}{\pi^3} 4(1+\pi L)^2\int_{1-\frac{\sqrt{3}}{L}}^{\infty} \frac{t^4}{(t-r_1)^6}dt = O(L^5) \, , \, L\rightarrow \infty.
\end{equation}
In the final inequality we've used the estimate
\begin{equation*}
t^2h(t)= \frac{t^4}{(t-r_1)^6}\left(c+c\frac{r_1}{t}+\frac{1}{t}\right)^2 \leq \frac{t^4}{(t-r_1)^6}\left(2c+2\right)^2 = \frac{t^4}{(t-r_1)^6}\left(2\pi L+2\right)^2 ,
\end{equation*}
which follows from $r_1<1-\frac{\sqrt{3}}{L}<t$ and $L\ge 8$. Integrating (\ref{eq:KsiPointEst}) over $B(0,r_1R)$ and dividing by $c^6$, we obtain (under the assumption that $c=\pi L$)
\begin{equation} \label{eq:NormKsi}
\frac{1}{c^6}\left|\left| \xi_c \right| \right|_{L^2(r_1R)}^2 = O(\frac{1}{L}) \, , \, L \rightarrow \infty.
\end{equation}
This estimate should be compared to (\ref{eq:ksic_bound}), where we've shown that $\left|\left|\xi_c\right|\right|_{L^2(R)}\leq \frac{4\pi}{3}c^{\frac{3}{2}}L^{\frac{3}{2}}$. The latter implies that $\left|\left|\xi_c\right|\right|_{L^2(r_1R)}\leq \frac{4\pi}{3}c^{\frac{3}{2}}L^{\frac{3}{2}}$, which gives us
\begin{equation*}
\frac{1}{c^6}\left|\left| \xi_c \right| \right|_{L^2(r_1R)}^2 \leq \frac{16\pi^2}{9}\frac{L^3}{c^3} = O(1)
\end{equation*}
under the assumption that $c=\pi L$.
\section{Almost orthogonality of GPSWFs sample vectors}
It is favourable if the vectors obtained by sampling the GPSWFs on the Cartesian grid are ``almost'' orthogonal (see \citep{landa2017steerable}, section 6 for a complete discussion in the 2D case). In this section, we verify numerically that for a truncation parameter $T\gg1$ this is indeed the case. Define the normalized GPSWFs by
\begin{equation*}
\hat{\psi}_{N,m,n} : = \left(\frac{c}{2\pi L}\right)^{\frac{3}{2}} \alpha_{N,n}\psi_{N,m,n},
\end{equation*}
where $\alpha_{N,n}$ is the eigenvalue corresponding to $\psi_{N,m,n}$, and define the Gram matrix $H_c$ by
\begin{equation*}
H_c := \left(\hat{\Psi}_c\right)^{*} \hat{\Psi}_c,
\end{equation*}
where $\hat{\Psi}_c$ is a matrix whose columns contain the samples of the normalized GPSWFs $\hat{\psi}_{N,m,n}$. We would like $H_c$ to be as close to the identity matrix as possible. Equivalently, we would like the eigenvalues of $H_c$ to be as close as possible to 1. In Figure~\ref{fig:Orthogonality}, we plot the maximal deviation (in absolute value) of the eigenvalues of $H_c$ from 1, that is,
\begin{equation*}
\max_{k}\left|\tau_k-1\right|,
\end{equation*}
where $\tau_1,\cdots,\tau_S$ are the eigenvalues of $H_c$. It is evident from Figure~\ref{fig:Orthogonality} that the deviation is proportional to $T^{-2}$.

\renewcommand\thefigure{\arabic{figure}} 
\begin{figure}
\begin{center}
\begin{tabular}{ccccclllll}
\includegraphics[width=110mm]{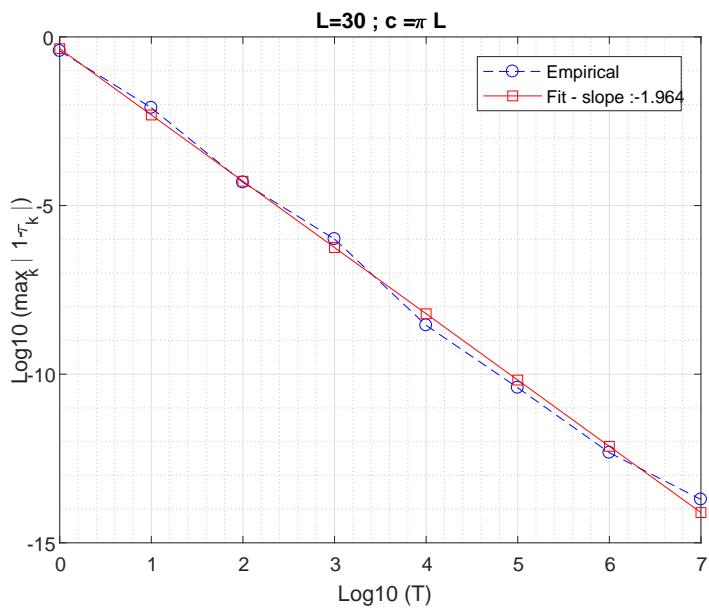}
\end{tabular}
\end{center}
\caption{Measured deviation of the eigenvalues of $H_c$ from 1 as a function of $T$.}
\label{fig:Orthogonality}
\end{figure}

\section*{Acknowledgements}
We would like to thank Roy Lederman for providing the source code of~\cite{Lederman} for computing the radial component of the GPSWFs.

This research was supported by the European Research Council
(ERC) under the European Union’s Horizon 2020 research and innovation programme (grant
agreement 723991 - CRYOMATH) and by Award Number R01GM090200 from the NIGMS.

\bibliography{pswfbib}

\end{document}